\documentclass[11pt]{article}

\usepackage[margin=1in]{geometry}

\usepackage{amsmath}
\usepackage{amssymb}
\usepackage{amsthm}
\usepackage{bm}
\usepackage{color}

\usepackage{hyperref}
\hypersetup{
    colorlinks=true,     % false: boxed links; true: colored links
    linkcolor=blue,      % color of internal links
    citecolor=blue,      % color of links to bibliography
    filecolor=blue,      % color of file links
    urlcolor=blue        % color of external links
}
\usepackage[capitalise,nameinlink]{cleveref}

\newtheorem{theorem}{Theorem}[section]
\newtheorem{proposition}[theorem]{Proposition}
\newtheorem{lemma}[theorem]{Lemma}
\newtheorem{definition}[theorem]{Definition}
\newtheorem{corollary}[theorem]{Corollary}

\newtheorem{example}[theorem]{Example}

\newcommand{\cS}{\mathcal{S}}

\newcommand{\psd}{\succeq}

\newcommand{\RR}{\mathbb{R}}

\title{Limitations on the expressive power of\\ convex cones without long chains of faces}
\author{James Saunderson\thanks{Department of Electrical and Computer Systems Engineering, Monash University, VIC 3800, Australia. 
Email: \texttt{james.saunderson@monash.edu}}}
\begin{document}
\maketitle

\begin{abstract}
A convex optimization problem in conic form involves minimizing a linear
functional over the intersection of a convex cone and an affine subspace.
In some cases, it is possible to replace a conic formulation using a certain cone, 
with a `lifted' conic formulation using another cone that is higher-dimensional, 
but simpler, in some sense. One situation in which this can be computationally
advantageous is when the higher-dimensional cone is a Cartesian product of many `low-complexity' cones,
such as second-order cones, or small positive semidefinite cones. 

This paper studies obstructions to a convex cone having a lifted
representation with such a product structure.  The main result says that whenever a convex
cone has a certain neighborliness property, then it does not have a lifted
representation using a finite product of cones, each of which has only short
chains of faces.  This is a generalization of recent work of Averkov
(`Optimal size of linear matrix inequalities in semidefinite approaches to
polynomial optimization', SIAM J.\ Appl.\ Alg.\ Geom., Vol.\ 3, No.\ 1, 128--151, 2019)
which considers only lifted representations using products of positive
semidefinite cones of bounded size.  Among the consequences of the main result is
that various cones related to nonnegative polynomials do not have lifted
representations using products of `low-complexity' cones, such as smooth cones,
the exponential cone, and cones defined by hyperbolic polynomials of low
degree.
\end{abstract}

% \begin{keywords}
% Conic optimization, lifts, cone factorization, hyperbolicity cone
% \end{keywords}
% 
% \begin{AMS}
%   52A20, 90C25, 52C07
% \end{AMS}

\section{Introduction}

A \emph{conic optimization problem} has the form
\begin{equation}
	\label{eq:conic} \textup{minimize}_{x} \langle c,x\rangle \quad\textup{subject to}\quad A(x) = b,\; x\in K
\end{equation}
where $K\subseteq \RR^n$ is a closed convex cone, $A:\RR^n\rightarrow \RR^m$ is a
linear map, and $c\in \RR^n$ and $b\in \RR^m$.  Different choices of convex
cones in \cref{eq:conic} lead to different classes of convex optimization
problems.  For instance, if $K=\RR_+^n$ is a nonnegative orthant, we obtain a
linear program; if $K$ is a finite Cartesian product of second-order cones
$\mathcal{Q} = \{(x,y,z)\in \RR^3: \sqrt{x^2+y^2}\leq z\}$, we obtain a
second-order cone program; if $K = \cS_+^k$ is the cone of $k\times k$ positive semidefinite
matrices, we obtain a semidefinite program.  
Computationally, conic
optimization problems are often easier to solve if 
$K= K_1\times \cdots \times K_m$ where each of the $K_i$ are convex
cones of `low-complexity'. One reason for this is that basic algorithmic
primitives related to the cone, such as computing the Euclidean projection of a point onto the cone, 
are separable across the factors and so are easily parallelizable.

It is natural, then, to try to understand which families of convex optimization problems can be expressed
in conic form with respect to a finite Cartesian product of `low-complexity' cones. 
The question can be phrased in geometric terms by using the notion of a lifted representation of a convex cone.
\begin{definition}[{\cite{gouveia2013lifts}}]
\label{def:lift}
If $C\subseteq \RR^n$ and $K\subseteq\RR^{d}$ are closed convex cones then $C$ has a \emph{$K$-lift} if
$C = \pi(K \cap L)$
where $\pi:\RR^d\rightarrow \RR^n$ is a linear map and $L\subseteq \RR^d$ is a linear subspace.
\end{definition}
If $C$ has a $K$-lift, then any conic optimization problem using the cone $C$ can be reformulated as a conic optimization 
problem using $K$. As such we are interested in understanding when a convex cone $C$ has a $K$-lift where $K$ is a 
Cartesian product of `low-complexity' cones. There are many examples where such reformulations are possible.
\begin{itemize}
	\item If $C$ is the set of symmetric positive semidefinite matrices that are sparse
	with respect to a chordal graph, then $C$ has a $\cS_+^{k_1}\times \cdots \times \cS_+^{k_\ell}$-lift
	where the $k_i$ are the sizes of the maximal cliques in the graph~\cite{agler1988positive,grone1984positive}. This observation 
	can be exploited to yield significant computational savings for semidefinite programs with sparse data~(see, e.g.,~\cite{vandenberghe2015chordal}).
	\item The cone of scaled diagonally dominant matrices~\cite{ahmadi2019dsos},
	and the cone of sums of nonnegative circuit polynomials~\cite{iliman2016amoebas}, both have second-order cone 
	lifts, or equivalently, $(\cS_+^2)^m$-lifts for some $m$.\footnote{The fact that sums of nonnegative circuit polynomials
	have $(\cS_+^2)^m$-lifts was established by Averkov~\cite{averkov2019optimal}.} 
	These have been used to give more tractable certificates of 
	nonnegativity for multivariate polynomials than those given by general sums of squares.
	\item Ben-Tal and Nemirovski's book~\cite{ben2001lectures} has many nontrivial examples of convex sets 
	that have second-order cone lifts.
	\item Cones of sums of arithmetic-mean geometric-mean functions (SAGE functions)~\cite{chandrasekaran2016relative}, 
	used to certify nonnegativity of signomial functions, have lifted representations using 
	a product of three-dimensional relative entropy cones (or, equivalently, exponential cones). 
\end{itemize}
This paper, on the other hand, explores \emph{limitations} 
on the expressive power of lifts using finite products of `low-complexity' cones.
Until recently, the only result in this direction was the fact that a convex cone 
has an $\RR_+^m$-lift for some finite $m$ if and only if it is a polyhedral cone.

In a breakthrough paper, Fawzi~\cite{fawzi2018representing} studied the
expressive power of second-order cone programming. Among his results is that 
$\cS_+^3$ has no second-order cone lift.
Fawzi identified the importance of certain neighborliness
properties of the positive semidefinite cone as an obstruction to constructing
second-order cone lifts. He also introduced a combinatorial approach to establish
results of this type, via Tur\'an's theorem.

Averkov~\cite{averkov2019optimal} subsequently significantly extended Fawzi's
approach.  He introduced the \emph{semidefinite extension degree} of a convex
cone---the smallest $k$ such that $C$ has an $(\cS_+^k)^m$-lift for some positive integer $m$.
Averkov established that convex cones with a certain $k$-neighborliness property (made precise in
\cref{def:nb}) have semidefinite extension degree at least
$k+1$.  To prove this result, Averkov used Ramsey's theorem for uniform
hypergraphs to provide the key combinatorial obstruction. 

The present work moves beyond studying lifts using products of positive semidefinite cones of bounded size.
Instead, we consider a notion of `low-complexity' that just depends on the face lattice of a cone.
In particular, we consider the expressive power of $K$-lifts where $K$ is a finite product
of cones, each of which has only short chains of faces. 
The main result
(\cref{thm:main}) gives an explicit obstruction to a cone
having a lifted representation of this type. The obstruction is the same as that considered by Averkov, 
and is based on the existence of arbitrarily large finite collections of extreme rays of the cone that satisfy 
a certain neighborliness property (see \cref{def:nb}).

\subsection{Chains of faces, neighborliness, and our main result}

\subsubsection{Chains of faces}
If $K$ is a closed convex cone, then a subset $F\subseteq K$ is a \emph{face} if 
$x,y\in K$ and $\alpha x + (1-\alpha)y \in F$ for some $\alpha\in (0,1)$ implies that $x,y\in F$.
Note that the empty set is always a face. A collection of faces $F_1,F_2,\ldots,F_\ell\subseteq K$ is a \emph{chain of length $\ell$} if 
$F_1 \subsetneq F_2 \subsetneq \cdots \subsetneq F_\ell$.
 For a convex cone $K$, define 
\[ \ell(K) = \textup{maximum length of a chain of nonempty faces of $K$}.\]
 This is well-defined because $\ell(K) \leq \dim(K)+1$ where $\dim(K)$ is the dimension of the span of $K$ 
(see \cref{sec:ced-bound}). Crucially, $\ell(\cdot)$ is monotone, in the sense that if 
$F$ is a face of $K$ and $F \subsetneq K$ then $\ell(F)< \ell(K)$.

This quantity appears naturally in the study of facial reduction algorithms for conic optimization problems
(see, e.g.,~\cite{pataki2013strong,waki2013facial}), and has been used to give an upper bound on the Carath\'eodory number of 
a closed, pointed convex cone~\cite{ito2017bound}. 
In this paper we will think of cones without any long chains of faces as being of low complexity. Indeed $\ell(K) = 1$
if and only if $K$ is a linear subspace, and $\ell(K) = 2$ if and only if $K$ is a closed halfspace.  

\subsubsection{Neighborliness properties}
In \cref{def:nb}, to follow, we describe the  cones for which we can show the non-existence of lifts. This is the same class 
of convex cones that is considered by Averkov~\cite{averkov2019optimal}. 
In the following definition, and throughout, we fix an inner product $\langle\cdot,\cdot\rangle$ 
on $\RR^n$ and denote the associated norm by $\|\cdot\|$. If $C\subseteq \RR^n$ is a set, 
then $C^* = \{f\;:\; \langle f,x\rangle \geq 0\;\;\textup{for all $x\in C$}\}$ is the 
corresponding dual cone.
A closed convex cone $C\subseteq \RR^n$ is \emph{pointed} if $C \cap (-C) = \{0\}$ and 
\emph{full-dimensional} if $\textup{span}(C) = \RR^n$. A convex cone is \emph{proper} 
if it is closed, pointed, and full-dimensional.
For a convex cone $C\subseteq \RR^n$, we denote the set of extreme rays of $C$
by $\textup{Ext}(C)$
and the set of normalized extreme rays by $\textup{ext}(C):=\{v\in \RR^n\;:\; \|v\|=1,\; \RR_+v\in \textup{Ext}(C)\}$. 
 \begin{definition}
\label{def:nb}
 	Let $C$ be a proper convex cone. If $V\subseteq \textup{ext}(C)$ is a subset of  
 	normalized extreme rays of $C$, then $C$ is \emph{$k$-neighborly with respect to $V$} if for every $k$-element subset
 	$W\subset V$, there is some linear functional $f_W\in C^*$ such that $\langle f_W,v\rangle > 0$ if 
 	$v\in V\setminus W$ and $\langle f_W,v\rangle = 0$ if $v\in W$. 
 \end{definition}
In this paper we will be particularly interested in the case where $C$ is $k$-neighborly with respect to 
arbitrarily large finite subsets of normalized extreme rays of $C$. 

\begin{example}[Positive semidefinite cone] 
\label{eg:psd}
The cone of $(k+1)\times (k+1)$ positive semidefinite symmetric matrices 
is $k$-neighborly with respect to 
 \begin{equation*}
 	V = \{v_iv_i^T/\|v_i\|^2\;:\; i\in \mathbb{N}\}\quad\textup{where}\quad v_i := (1,i,i^2,\ldots,i^k). 
 \end{equation*}
 To see why this is true, for each set $W$ of $k$ natural numbers, define the non-negative polynomial $p_W$ and the vector 
 $c(W)\in \RR^{k+1}$ such that 
 \[ p_W(t) = \left[\prod_{i\in W}(t-i)\right]^2 = \left(\sum_{j=0}^{k}c(W)_j t^j\right)^2 = \textup{tr}(c(W)c(W)^T v_tv_t^T).\]
Observe that $p_W(t)$ vanishes if and only if $t\in W$. 
Consequently $c(W)c(W)^T\in (\cS_+^{k+1})^* = \cS_+^{k+1}$ vanishes on $v_iv_i^T$ if and only if $i\in W$. 
\end{example}
We are now in a position to state our main result.  
\begin{theorem}
\label{thm:main}
Let  $C$ be a proper convex cone, let $m$ be a positive integer, and let 
$K_1,\ldots,K_m$ be closed convex cones with $\ell(K_i)\leq k+1$ for each $i=1,2,\ldots,m$.
Suppose that, for each $N\in \mathbb{N}$ there exists a subset $V_N\subseteq \textup{ext}(C)$
such that $C$ is $k$-neighborly with respect to $V_N$ and $|V_N| \geq N$. 
Then $C$ does not have a $K_1\times \cdots \times K_m$-lift.
\end{theorem}
We remark that if a proper convex cone $C$ is $k$-neighborly with respect to an infinite set $V\subseteq \textup{ext}(C)$, then 
$C$ is $k$-neighborly with respect to any finite subset of $V$, and so the neighborliness hypothesis of \cref{thm:main} holds.

\Cref{sec:nb-eg} gives examples of convex cones that are $k$-neighborly with respect to 
arbitrarily large finite sets of normalized extreme rays. 
\Cref{sec:ced-bound} gives examples of convex cones with only short chains of faces.
Combining these allows us to specialize \cref{thm:main} to produce a range of irrepresentability results. 
The following statements about the non-existence of lifts of certain positive semidefinite cones
are examples of the kind of results that follow from \cref{thm:main}: 
\begin{itemize}
	\item $\cS_+^3$ has no $K$-lift where $K$ is a finite Cartesian product of smooth convex cones (such as 
	second-order cones);
	\item $\cS_+^4$ has no $K$-lift where $K$ is a finite Cartesian product of three-dimensional convex cones
	(such as cones over the epigraphs of univariate convex functions);
	\item $\cS_+^{k+1}$ has no $(\cS_+^{k})^m$-lift where $m$ is a positive integer, a result from~\cite{averkov2019optimal};
	\item $\cS_+^{k+1}$ has no $K$-lift where $K$ is a hyperbolicity cone corresponding to a hyperbolic 
	polynomial with all its irreducible factors having degree at most $k$. 
\end{itemize} 

Our proof of \cref{thm:main} follows Averkov's approach to
lower-bounding the semidefinite extension degree.  In fact, the underlying
combinatorial part of the argument is exactly the same.  The main new
contribution is that all of the algebraic structure of the positive
semidefinite cone used by Averkov can be done away with, and replaced with
basic properties of face lattice of convex cones.

\subsection{Notation}
For a convex cone $C$, let 
$\textup{relint}(C)$ denote the relative interior of $C$. If $S\subseteq \RR^n$ let $\textup{cone}(S)$ be the cone generated 
by $S$, i.e., the collection of all non-negative combinations of elements of $S$. 
If $n$ is a positive integer let $[n] := \{1,2,\ldots,n\}$. If $S$ is a finite set 
let $\binom{S}{n}$ be the collection of $n$-element subsets of $S$. Other notation, needed only for the proofs, will be 
introduced in \cref{sec:pre}. 

\subsection{Outline}

The rest of the paper is structured as follows. In \cref{sec:con} we discuss the consequences of \cref{thm:main}. 
We state a number of examples (many from Averkov's work) of cones that are $k$-neighborly with respect to an infinite set.
We then give bounds on the length of chains of faces for a number of families of convex cones. From the discussion one can extract
many irrepresentability results. For the sake of brevity we will not exhaustively state such results. 
In \cref{sec:pre} we briefly introduce some of the technical tools used 
in the proof of the main result. In \cref{sec:pf} we generalize the key technical results of~\cite{averkov2019optimal}
to our setting, prove these results, and consequently complete the proof of \cref{thm:main}.

\section{Consequences of the main result}
\label{sec:con}
To appreciate the consequences of our main result, this section is devoted to
giving examples of convex cones to which \cref{thm:main} can be applied. In
\cref{sec:nb-eg} we give examples of convex cones  that are $k$-neighborly (for
some $k$) with respect to arbitrarily large finite sets of normalized extreme
rays.  These are the convex cones that can be used as $C$ in \cref{thm:main}.
In \cref{sec:ced-bound} we give upper bounds on $\ell(\cdot)$, for many convex
cones. These are the cones that can be used as the $K_i$ in \cref{thm:main}.

\subsection{Cones $k$-neighborly with respect to arbitrarily large finite sets}
\label{sec:nb-eg}

\subsubsection{Non-polyhedral cones}
If $C$ is a proper, non-polyhedral cone then it has infinitely many extreme rays. Moreover, it has infinitely many \emph{exposed} 
extreme rays (extreme rays that can be obtained as the intersection of $C$ with a hyperplane), because 
exposed extreme rays are dense in the set of all extreme rays by Straszewicz's theorem~\cite[Theorem 18.6]{rockafellar2015convex}. 
It follows that $C$ is $1$-neighborly with 
respect to the (infinite) set of normalized exposed extreme rays.

\subsubsection{Cones related to nonnegative polynomials and sums of squares} 
A number of interesting examples are special cases of the following result, which is essentially~\cite[Corollary 3]{averkov2019optimal}.
\begin{proposition}
\label{prop:Cav}
Let $X\subseteq \RR^n$ have nonempty interior. Define 
	\begin{align*}
		P_{n,2d}(X) & := \{\textup{polynomials $p$ of degree $\leq 2d$ such that $p(x)\geq 0$ for all $x\in X$}\}\\
		\Sigma_{n,2d} & := \{\textup{polynomials of degree $\leq 2d$ in $n$ variables that are sums of squares}\}.
	\end{align*}
	Let $C$ be a closed convex cone that satisfies 
	$P_{n,2d}(X)^*\subseteq C \subseteq \Sigma_{n,2d}^*$. Then for each $N\in \mathbb{N}$ there is a 
	set $V\subseteq \textup{ext}(C)$ with $|V|\geq N$ such that $C$ is $\left(\binom{n+d}{d}-1\right)$-neighborly 
	with respect $V$. 
\end{proposition}
From~\cref{prop:Cav} we can construct many examples. 
Perhaps the simplest, which are also discussed in Averkov's paper, are the following:
\begin{itemize}
	\item The (self-dual) cone of $k\times k$ positive semidefinite matrices $\cS_+^k$
is linearly isomorphic to $\Sigma_{k-1,2}$, since a quadratic polynomial in $k-1$ variables is nonnegative 
if and only if it has the form 
\[ q(x) = \begin{bmatrix} 1& x^T \end{bmatrix}Q\begin{bmatrix}1\\x\end{bmatrix}\]
where $Q$ is positive semidefinite.
	It then follows from \cref{prop:Cav} that $\cS_+^k$ is $(k-1)$-neighborly with respect to arbitrarily large finite
	subsets of normalized extreme rays. We gave a direct argument that $\cS_+^k$ is $(k-1)$-neighborly with respect to 
	an infinite subset of extreme rays in~\cref{eg:psd}.
	\item The dual of the cone of univariate nonnegative polynomials of degree at most $2d$, i.e., $P_{1,2d}(\RR)^*$ 
	is $d$-neighborly with respect to arbitrarily large finite subsets of normalized extreme rays.
\end{itemize}

\subsubsection{Cones over $k$-neighborly manifolds} If $\mathcal{M}\subseteq \RR^m$ is a smooth embedded manifold, 
Kalai and Widgerson~\cite{kalai2008neighborly} say that $\mathcal{M}$ is \emph{$k$-neighborly}
if, for every subset $X\subseteq \mathcal{M}$ of cardinality $k$, there exists a linear functional $f$, and a real number $b$,
such that
$\langle f,x\rangle = b$ for all $x\in X$ and $\langle f,x\rangle > b$ for all $x\in \mathcal{M}\setminus X$. 
If $\mathcal{M}$ is a $k$-neighborly manifold of positive dimension, then
$\textup{cone}(\mathcal{M}\times \{1\})\subseteq \RR^{m+1}$ is $k$-neighborly with respect to the 
infinite set $\{v/\|v\|\;:\; v\in \mathcal{M}\times \{1\}\}$ of normalized extreme rays.

\subsubsection{Cones with neighborly faces and derivative relaxations of $\cS_+^k$}
Suppose that $C$ is a convex cone and $F$ is an exposed face of $C$. One can show that if $F$ is $k$-neighborly with respect 
to $V\subseteq \textup{ext}(F)\subseteq \textup{ext}(C)$ then $C$ is also $k$-neighborly with respect to $V$. 
This simple observation gives some interesting additional examples beyond cones related to nonnegative polynomials. 
We describe just one example here, which is an example of a \emph{hyperbolicity cone} (a class of convex cones we discuss more
at the end of \cref{sec:ced-bound}).

The cone of $k\times k$ positive semidefinite matrices can be described as
\[ \cS_+^k = \{X\in \cS^k\;:\; E_{1,k}(X) \geq 0,\; E_{2,k}(X) \geq 0,\; \ldots,\; E_{k,k}(X) \geq 0\}\]
where $E_{\ell,k}(X)$ is the sum of the $\ell\times \ell$ principal minors of the $k\times k$ symmetric  matrix $X$
(see, e.g.,~\cite{renegar2006hyperbolic}).
For $0\leq \ell \leq k-1$ one can define a cone (which turns out to be convex) by keeping only the first $k-\ell$ of these inequalities:
\[ \cS_+^{k,(\ell)} = \{X\in \cS^k\;:\; E_{1,k}(X) \geq 0,\; \ldots,\; E_{k-\ell,k}(X) \geq 0\}.\]
These are often called the \emph{derivative relaxations} or \emph{Renegar derivatives} of the positive semidefinite cone.
For the connection to derivatives, and the fact that these are convex cones, see~\cite{renegar2006hyperbolic}. 
Consider the intersection of $\cS_+^{k,(\ell)}$ with the subspace $L$
of symmetric matrices of the form $\left[\begin{smallmatrix} Y & 0\\0 & 0\end{smallmatrix}\right]$
where only the top $(k-\ell)\times (k-\ell)$ block is nonzero. Restricted to this subspace we have that 
\[ E_{p,k}\left(\begin{bmatrix} Y & 0\\0 & 0\end{bmatrix}\right) = 
	\begin{cases} E_{p,k-\ell}(Y) & \textup{if $p\leq k-\ell$}\\0 & \textup{otherwise.}\end{cases}\]
As such $\cS_+^{k,(\ell)}\cap L$ is a face of $\cS_+^{k,(\ell)}$ that is linearly isomorphic
to $\cS_+^{k-\ell}$. Since $\cS_+^{k-\ell}$ is $(k-\ell-1)$-neighborly with respect to an infinite set of normalized 
extreme rays, it follows that $\cS_+^{k,(\ell)}$ has the same property.

\subsection{Cones with only short chains of faces}
\label{sec:ced-bound}

The innovation of \cref{thm:main} is that it rules out $K$-lifts where $K$ is a finite Cartesian product of 
cones, each of which has only short chains of faces. 
In this section we now give a number of examples of cones $K$ that have only short chains of faces.

\subsubsection{Rays} The ray $\RR_+$ has two nonempty faces: $\{0\}$ and $\RR_+$ itself. It follows that $\ell(\RR_+) = 2$.
From \cref{thm:main} we recover the fact that any cone that is $1$-neighborly with respect to an infinite set of normalized 
extreme rays (i.e., a cone with infinitely many exposed extreme rays) cannot have a $\RR_+^m$-lift. 

\subsubsection{Smooth cones}
Following~\cite{liu2018exact}, for instance, we call a pointed closed convex cone $K$ \emph{smooth} if its only non-empty faces 
are $\{0\}$ or $K$ or its extreme rays. As such, any smooth cone has $\ell(K) \leq 3$. 
For example, the second-order cone
\[ \mathcal{Q}_{n+1} = \{(x_0,\ldots,x_{n})\in \RR^{n+1}\;:\; \sqrt{x_1^2 + \cdots + x_n^2} \leq x_0\}\]
is a smooth cone. From \cref{thm:main} we can conclude that any cone $2$-neighborly with respect to arbitrarily large finite
sets of normalized extreme rays does not have a lift using a finite product of smooth cones.

\subsubsection{Low-dimensional cones}
Since the dimension strictly increases along chains of 
faces~\cite[Corollary 18.1.3]{rockafellar2015convex}, we always have that 
\[ \ell(K) \leq \textup{dim}(K)+1.\]
For example, the exponential cone $K = \textup{cl}\{(x,t,y)\in \RR^2\times \RR_{++}: ye^{x/y} \leq t\}$ satisfies $\ell(K) \leq 3+1$. 
More generally, if $K = \textup{cl}\{(x,t,y)\in \RR^2\times \RR_{++}: y\,g(x/y) \leq t\}$ is the closure of the cone over the epigraph 
of any univariate convex function $g$, then $\ell(K) \leq 3+1$. 
From \cref{thm:main} we can conclude that any cone $k$-neighborly with respect to arbitrarily large finite sets of normalized 
extreme rays 
does not have a lift using a finite product of $k$-dimensional cones.

\subsubsection{The positive semidefinite cone}
The rank of a symmetric matrix is constant on the relative interior of faces of the positive semidefinite cone. 
Moreover, the rank function strictly increases on chains of faces. As such,
\[ \ell(\cS_+^k) \leq k+1.\]
On the other hand, it is easy to construct a chain of non-empty faces of length $k+1$, so we have that $\ell(\cS_+^k) = k+1$.
\Cref{thm:main}, in this context, recovers Averkov's result that if $C$ is $k$-neighborly with respect to arbitrarily large
finite sets of normalized extreme rays then $C$ has no $(\cS_+^k)^m$-lift.

\subsubsection{Hyperbolicity cones}
A class of convex cones that generalizes the positive semidefinite cone are hyperbolicity cones. 
A homogeneous polynomial $p$ of degree $d$ in $n$ variables with real coefficients 
is called \emph{hyperbolic with respect to $e\in \RR^n$} if 
$p(e) > 0$ and, for all $x\in \RR^n$, the univariate polynomial $t\mapsto p(te-x)$ has only real roots. 
We call these real roots the \emph{hyperbolic eigenvalues of $x$}. There is an associated closed cone
\[ \Lambda_+(p,e) = \{x\in \RR^n\;:\; \textup{all hyperbolic eigenvalues of $x$ are nonnegative}\}\]
which is convex, a result of G{\aa}rding~\cite{gaarding1959inequality}.
As an example, the determinant restricted to symmetric matrices is hyperbolic with respect to the identity matrix, and the 
associated hyperbolicity cone is the positive semidefinite cone. 

With a hyperbolic polynomial, one can associate a rank function by defining
\[ \textup{rank}_{p,e}(x) = \textup{\# non-zero hyperbolic eigenvalues of $x$}.\]
Renegar~\cite[Theorem 26]{renegar2006hyperbolic} has showed that if $F$ is a face of a hyperbolicity cone $\Lambda_+(p,e)$ then 
every point in the relative interior of $F$ has the same hyperbolic rank, and every point in the boundary 
of $F$ has strictly smaller rank. Consequently, there is a well defined notion of the hyperbolic rank of a face, and 
the hyperbolic rank function is strictly increasing along chains of faces.
Since any point in the relative interior of the hyperbolicity cone has rank $d = \textup{degree}(p)$, we have the bound
\[ \ell(\Lambda_+(p,e)) \leq d + 1.\]
The following gives a natural sufficient condition under which this bound is tight.
\begin{proposition}
\label{prop:hyprkone}
Suppose that $p$ is homogeneous of degree $d$, hyperbolic with respect to $e$, 
and the associated hyperbolicity cone $\Lambda_+(p,e)$ is pointed. If all of the extreme rays of $\Lambda_+(p,e)$ have
hyperbolic rank one then $\ell(\Lambda_+(p,e)) = d+1$.
\end{proposition}
\begin{proof}
Let $x_1,x_2,\ldots,x_{\kappa}$ be a collection of generators of extreme rays such that $x_1+x_2+\cdots + x_{\kappa}$
is in the relative interior of $\Lambda_+(p,e)$. Since the Carath\'eodory number of $\Lambda_+(p,e)$ is bounded above by 
$\ell(\Lambda_+(p,e)) - 1$~\cite{ito2017bound}, it follows that we can choose $\kappa\leq \ell(\Lambda_+(p,e))-1$. 
Since the hyperbolic rank function is a nonnegative submodular function on the face lattice of the 
hyperbolicity cone~\cite{amini2018non}, the hyperbolic rank is subadditive. Then
\[ \ell(\Lambda_+(p,e)) -1 \leq d = \textup{rank}_{p,e}\left(\sum_{i=1}^{\kappa}x_i\right) \leq \sum_{i=1}^{\kappa}\textup{rank}_{p,e}(x_i)
= \kappa \leq \ell(\Lambda_+(p,e))-1.\]
\end{proof}
Symmetric cones\footnote{Self-dual conex cones for which the automorphism group acts
transitively on the interior.} are examples of hyperbolicity cones with $p$
being the determinant associated with the appropriate Euclidean Jordan algebra.
For these cones, all of the extreme rays have hyperbolic rank one.  As
such \cref{prop:hyprkone} generalizes~\cite[Theorem 14]{ito2017bound}. It also applies
to spectrahedra with all rank one extreme rays which, in the real symmetric
case, have been classified by Blekherman, Sinn, and
Velasco~\cite{blekherman2017sums}. 

We can obtain a more refined bound in cases where all the extreme rays 
have rank at least $r$, and so all nonzero faces have rank at least $r$. Then 
\[ \ell(\Lambda_+(p,e)) \leq d - r +2\]
since, in that case, there are no faces of rank $1,2,\ldots,r-1$. For hyperbolicity cones corresponding to 
strictly hyperbolic polynomials~\cite{nuij1969note} (which give rise to smooth hyperbolicity cones), all extreme rays
have rank $d-1$ so this bound becomes $\ell(\Lambda_+(p,e)) \leq 3$.

A hyperbolic polynomial is \emph{irreducible} if it cannot be factored as a product of  
polynomials of lower degree. If $p$ is hyperbolic with respect to $e$ and is reducible, 
we can write $p = p_1^{m_1}\cdots p_n^{m_n}$ where the $p_i$ are the irreducible factors of $p$
and the $m_i$ are positive integers. It is straightforward to see that the $p_i$ are also hyperbolic with respect to $e$. 
The hyperbolicity cone corresponding to $p$ is the intersection of the cones corresponding to the $p_i$. As such, 
\[ \Lambda_+(p,e) = \{x\;:\; (x,x,\ldots,x) \in  \Lambda_+(p_1,e)\times \cdots \times \Lambda_+(p_n,e)\}\]
gives a $\Lambda_+(p_1,e)\times \cdots \times \Lambda_+(p_n,e)$-lift of $\Lambda_+(p,e)$. 

The following corollary of \cref{thm:main} summarizes some of the discussion above, and highlights 
one of very few known obstructions to the existence of lifts using hyperbolicity cones. 
\begin{corollary}
\label{cor:hyp}
Suppose $C$ is a proper convex cone such that for any $N\in \mathbb{N}$ there is a subset
$V_N\subseteq \textup{ext}(C)$ such that $|V_N|\geq N$ and $C$ is   
$k$-neighborly with respect to $V_N$.
If $p$ is hyperbolic with respect to $e$ and all the irreducible components of $p$ have degree at most $k$, 
then $C$ does not have a $\Lambda_+(p,e)$-lift.
\end{corollary}

\section{Technical preliminaries}
\label{sec:pre}
In this section we briefly introduce the additional definitions, and 
basic technical results, needed for our proof of \cref{thm:main}. 

\subsection{Convex cones and their faces}

The \emph{lineality space} $\textup{Lin}(C)$ of a closed convex cone is $\textup{Lin}(C) := C \cap (-C)$, and is the largest
subspace contained in $C$. A closed convex cone $C$ is pointed if and only if $\textup{Lin}(C) = \{0\}$. 
A closed convex cone can always be expressed as $C = C' + \textup{Lin}(C)$ where $C'$ is pointed and the sum is direct 
(for instance we can take $C' = C \cap \textup{Lin}(C)^\perp$).

If $C\subseteq \RR^n$ is a closed convex cone and 
$X\subseteq C$, we denote by $F_C(X)$ the smallest (inclusion-wise) face of $C$ containing $X$.  
If $x\in C$, we write $F_C(x)$ instead of $F_C(\{x\})$ for the smallest (inclusion-wise) face of $C$ containing $x$. 
If $C$ is clear from the context, we omit it from the notation.

The collection of faces of a proper convex cone $C$, 
partially ordered by inclusion, form a lattice~\cite{barker1973lattice}.
The lattice operations are 
\[ F_1\wedge F_2 = F_1 \cap F_2\quad\textup{and}\quad
F_1\vee F_2 = F_C(F_1 \cup F_2).\]

We summarize some useful properties of this operation for future reference.
\begin{lemma}
\label{lem:face-arithmetic}
Let $C$ be a proper convex cone.
\begin{enumerate}
	\item If $F$ is a face of $C$ and $F\subsetneq C$ then $\dim(F) < \dim(C)$. 
	\item If $x\in C$ and $F$ is a face of $C$ then $F_C(x) = F$ if and only if $x\in \textup{relint}(F)$.
	\item If $\lambda > 0$ and $x\in C$ then $F_C(\lambda x) = F_C(x)$.
	\item If $x,y\in C$ then $F_C(x+y) = F_C(x)\vee F_C(y)$.
	\item If $x_1,\ldots,x_n\in C$ and $\lambda_1,\ldots,\lambda_n > 0$ then 
\[ F_C\left(\sum_{i=1}^{n}\lambda_i x_i\right) = \bigvee_{i=1}^{n}F_C(\lambda_ix_i) = \bigvee_{i=1}^{n}F_C(x_i) = 
F_C\left(\sum_{i=1}^{n}x_i\right).\]
\end{enumerate}
\end{lemma}
\begin{proof}
The first statement is~\cite[Corollary 18.1.3]{rockafellar2015convex}. 
For the second statement, suppose that $x\in \textup{relint}(F)$. 
Then $F$ is a face of $C$ that contains $x$, so $F_C(x)\subseteq F$. Moreover, $x\in F_C(x)\cap \textup{relint}(F)$ so, 
by~\cite[Corollary 18.1.2]{rockafellar2015convex}, $F \subseteq F_C(x)$. Now assume that $F_C(x) = F$. Then, 
by~\cite[Lemma 2.9]{barker1973lattice}, $x\in \textup{relint}(F_C(x)) = \textup{relint}(F)$.
The third statement follows from the fact that if $F$ is a face of $C$ then $x\in F$ if and only if $\lambda x \in F$.
The fourth statement is a special case of~\cite[Corollary 1]{barker1975cones}.
The fifth statement follows from statements 3 and 4. 
\end{proof}

\subsection{Cone lifts and the factorization theorem}
We begin with the notion of a proper\footnote{The notion of a proper lift of a closed convex cone 
(from \cref{def:proper-lift}) is quite distinct from the notion of a proper (i.e., closed, pointed, full-dimensional) 
convex cone. The distinction between these (standard) uses should be clear from the context.}
lift, a slight refinement of \cref{def:lift}.
\begin{definition}
\label{def:proper-lift}

If $C\subseteq \RR^n$ and $K\subseteq\RR^{d}$ are closed convex cones then $C$ has a \emph{proper $K$-lift} if
$C = \pi(K \cap L)$
where $\pi:\RR^d\rightarrow \RR^n$ is a linear map and $L\subseteq \RR^d$ is a linear subspace that meets the 
relative interior of $K$.
\end{definition}
The factorization theorem~\cite[Theorem 1]{gouveia2013lifts}
of Gouveia, Parrilo, and Thomas, allows us to reformulate geometric questions about the existence of cone lifts 
as algebraic questions about the existence of factorizations of certain non-negative operators. 
Here we state and prove (for completeness) only the direction of the result of Gouveia, Parrilo, and Thomas that we need,
and express it in a  modified form that is most convenient for our subsequent use. 
\begin{theorem}[{Factorization theorem~\cite[Theorem 1]{gouveia2013lifts}}]
\label{thm:gpt}
If a proper convex cone  $C$ has a proper $K_1\times \cdots\times K_m$-lift then, for $i=1,2,\ldots,m$, 
there are maps 
$a_i: C^* \rightarrow K_i^*$
and 
$b_i: C \rightarrow K_i$ such that 
$\langle x,y\rangle = \sum_{i=1}^{m} \langle b_i(x),a_i(y)\rangle$
for all $(x,y)\in C\times C^*$. 
\end{theorem}
\begin{proof}
Since $C$ has a proper $K:= K_1\times \cdots \times K_m$-lift there is a
subspace $L$ that meets the relative interior of $K$ such that $C =\pi(K \cap
L)$.  For each $x\in C$ define $b(x)$ to be an arbitrary choice of element of
$K \cap L$ such that $\pi(b(x)) = x$. 

Since $C = \pi(K\cap L)$ and $L$ meets the relative interior of $K$, 
it follows from~\cite[Corollary 16.4.2]{rockafellar2015convex} 
that $\pi^*(C^*)\subseteq (K \cap L)^* = K^* + L^\perp$.
So, for each $y\in C^*$,
there exists $a(y)\in K^*$ and $w(y)\in L^\perp$ such that $\pi^*(y) = a(y) + w(y)$. 
Then 
\[ \langle x,y\rangle = \langle \pi(b(x)),y\rangle = \langle b(x),\pi^*(y)\rangle = \langle b(x),a(y) + w(y)\rangle = \langle b(x),a(y)\rangle\]
since $b(x)\in L$ and $w(y)\in L^\perp$. If we let the $a_i$ and $b_i$ be the associated coordinate functions of $a$ and $b$, the 
result follows. 
\end{proof}
The following result means we can replace general lifts with proper lifts when we prove \cref{thm:main},
allowing us to use the factorization theorem.
\begin{lemma}
\label{lem:proper-ced}
	Suppose that $C$ is a proper convex cone and there exist a positive integer $m$ and 
	closed convex cones $\tilde{K}_1,\tilde{K}_2,\ldots,\tilde{K}_m$ such that $C$ has 
	a $\tilde{K}_1\times \cdots \times \tilde{K}_m$-lift and 
	$\ell(\tilde{K}_i) \leq k+1$ for all $i\in [m]$. Then there exist a positive integer $m$ and closed convex cones 
	$K_1,K_2,\ldots,K_m$ 
	such that $C$ has a proper $K_1\times \cdots \times K_m$-lift and $\ell(K_i) \leq k+1$ for all $i\in [m]$.
\end{lemma}
\begin{proof}
	By our assumption on $C$, there is a linear map $\pi$ and a subspace $L$ such that
	$C = \pi(L \cap (\tilde{K}_1\times \cdots \times \tilde{K}_m))$. If $L$ meets the relative interior of 
	$\tilde{K}_1\times \cdots \times \tilde{K}_{m}$ we simply take $K_i = \tilde{K}_i$ for all $i\in [m]$. 
	 If $L$ does not intersect the relative interior of $\tilde{K}_1\times \cdots \times \tilde{K}_{m}$
	then let $F$ denote the minimal face of $\tilde{K}_1\times \cdots \times \tilde{K}_{m}$
	 containing $(\tilde{K}_1\times \cdots \times \tilde{K}_{m})\cap L$.
	Since $F$ is a face of $\tilde{K}_1\times \cdots \times \tilde{K}_m$ it has the form $F_1\times \cdots \times F_m$ 
	where $F_i$ is a face
	of $\tilde{K}_i$ for all $i\in [m]$~\cite[Theorem 2]{barker1978perfect}. 
	Moreover, we have that $\ell(F_i) \leq \ell(\tilde{K}_i)$ for all $i\in [m]$. As such, $C$ has a proper
	$F_1\times F_2\times \cdots \times F_m$-lift and $\ell(F_i) \leq k+1$ for all $i\in [m]$. Taking $K_i = F_i$ for all $i\in [m]$ 
	completes the proof.
\end{proof}	

\subsection{Ramsey numbers}
The key combinatorial result we use is Ramsey's theorem for hypergraphs. 
We formally state it in language more suited to our application.
\begin{theorem}[{\cite[Theorem B]{ramsey1930problem}}]
Let $k$, $m$, and $n$ be positive integers. There exists a positive integer $R_k(m;n)$ such that  
whenever $V$ is a set with $|V|\geq R_k(m;n)$ and $f:\binom{V}{k}\rightarrow [n]$ then there exists 
$W\subseteq V$ such that $|W|=m$ and $f(S) = f(T)$ for all $S,T\in \binom{W}{k}$.
\end{theorem} 
In the hypergraph setting we think of $f$ as a coloring of the $k$-uniform hypergraph on vertex set $V$ with $n$ colors, 
and $W$ as a subset of $m$ vertices of the hypergraph for which all induced hyperedges have the same color.
Except in a few very special cases, the numbers $R_k(m;n)$ are not known, although explicit upper bounds
are available. For our purposes, we only need the fact that $R_{k}(m;n)$ is finite for positive integers $n, k$, and $m$. 

\section{Generalizing Averkov's lemmas and the proof of \cref{thm:main}}
\label{sec:pf}
In this section we establish \cref{thm:main} by generalizing the key technical arguments of~\cite{averkov2019optimal}.
Crucial to Averkov's approach is the following simple result about positive semidefinite matrices. In the statement, if $X$
is a positive semidefinite matrix, let $\textup{col}(X)$ denote its column space, and interpret the empty sum as zero, i.e., 
$\sum_{i\in \emptyset}X_i = 0$.
\begin{lemma}
\label{lem:averkov-subset}
Suppose $X_1,X_2,\ldots,X_\ell\psd 0$ and $\textup{rank}(\sum_{i\in[n]}X_i) = k$. Then 
there exists a subset $I\subseteq [n]$ with $|I|\leq k$ such that $\textup{col}\left(\sum_{i\in I}X_i\right) = 
\textup{col}\left(\sum_{i\in [n]} X_i\right)$.
\end{lemma}
We note that Averkov states the conclusion in an equivalent form  
as $\sum_{i\in I}\textup{col}(X_i) = \sum_{i\in [n]}\textup{col}(X_i)$. The statement given here suggests, more clearly, 
the generalization we require. 

\Cref{lem:averkov-subset} can be generalized to the setting in which the rank is 
replaced with the largest length of a chain of nonempty faces (minus one), and the column space is 
replaced with the minimal face containing a point of the convex cone. In the statement and proof
of \cref{lem:js-subset} we interpret the empty sum as zero, i.e., $\sum_{i\in \emptyset}x_i = 0$.
\begin{lemma}
\label{lem:js-subset}
Let $K$ be a closed 
convex cone and suppose that $x_1,x_2,\ldots,x_{n}\in K$ are such 
that $\sum_{i=1}^{n}x_i\in \textup{relint}(K)$. Then there exists 
$I\subseteq [n]$ with $|I|\leq \ell(K)-1$ such that $F(\sum_{i\in I}x_i) = F(\sum_{i\in [n]}x_i) = K$.
\end{lemma}
\begin{proof}
First, we will prove the result under the assumption that $K$ is pointed.
We argue by induction on $\ell(K)$. For the base case, consider a closed pointed convex cone
with $\ell(K) = 1$. The only possibility is $K= \{0\}$. In this case if $x_1,\ldots,x_n\in K$
we can choose $I = \emptyset$ so that $|I| = 0 = \ell(\{0\})-1$ and $F(\sum_{i\in I}x_i) = F(0) = \{0\} = K$. 

Assume the statement holds for all closed pointed convex cones $K$ with $\ell(K)\leq k$, for some positive integer $k$.
Consider a closed pointed 
convex cone $C$ with $\ell(C) = k+1$. Let $x_1,\ldots,x_n\in C$ be such that $F(x_1+\cdots+x_n) = C$. 
Let $I\subseteq [n]$ be an inclusion-wise minimal subset of $[n]$ such that $F(\sum_{i\in I}x_i) = C$. 
Choose some $j\in I$ and observe that $x_j \notin F(\sum_{i\in I\setminus\{j\}} x_i)$ (by minimality of $I$)
and so $F(\sum_{i\in I\setminus\{j\}}x_i)$ is strictly contained in $C$. Hence,
\[ \textstyle{\ell(F(\sum_{i\in I\setminus\{j\}}x_i)) \leq \ell(C)-1 = k}.\]
Since $F(\sum_{i\in I\setminus\{j\}}x_i)$ is closed, convex, and pointed, by the induction hypothesis, 
there is some $I'\subseteq I\setminus\{j\}$ such that $|I'|\leq k-1$ and 
\[ \textstyle{F(\sum_{i\in I'}x_i) = F(\sum_{i\in I\setminus\{j\}}x_i).}\]
Then, by \cref{lem:face-arithmetic},
\begin{align*}
	\textstyle{ F(x_j + \sum_{i\in I'}x_i)} &=\textstyle{ F(x_j)\vee F(\sum_{i\in I'}x_i)}\\
 &\textstyle{= 
F(x_j)\vee F(\sum_{i\in I\setminus \{j\}}x_i) = F(\sum_{i\in I}x_i)=C.}\end{align*}
Since $I$ is minimal with this property and $I'\cup\{j\}\subseteq I$ 
we must have that $I'\cup\{j\}= I$. It then follows that $|I| = |I'|+1\leq k$, as required.

If $K$ is not pointed, then $K= K' + \textup{Lin}(K)$ where $K'$ is pointed and the sum is direct. Furthermore, 
$F'$ is a face of $K'$ if and only if $F'+\textup{Lin}(K)$ is a face of $K$. 
If $x = x'+z$ with $x'\in K'$ and $z\in \textup{Lin}(K)$ then $x\in \textup{relint}(K)$ if and only if $x'\in \textup{relint}(K')$.

Now, let $x_i = x_i' + z_i$ be decompositions of each of the $x_i$ so that $x_i'\in K'$ and $z_i\in \textup{Lin}(K)$ for all $i$. 
By assumption, $\sum_{i=1}^{n}x_i\in \textup{relint}(K)$ and so $\sum_{i=1}^{n}x_i'\in \textup{relint}(K')$. 
Since the lemma holds for the pointed case, 
there exists $I\subseteq [n]$ with $|I|\leq \ell(K')-1 = \ell(K)-1$ such that $\sum_{i\in I}x_i'\in \textup{relint}(K')$. 
Then $\sum_{i\in I}x_i \in \textup{relint}(K)$ as required.  
\end{proof}

Averkov's main lemma is the following result, which relies crucially on \cref{lem:averkov-subset} for its proof.
\begin{lemma}
\label{lem:averkov-main}
Let $m$ be a positive integer and let $S$ denote a finite set with cardinality at least $k$. Suppose that, for each $i\in [m]$,
there are maps $a_i:\binom{S}{k}\rightarrow \cS_+^k$ and $b_i:S\rightarrow \cS_+^k$
such that 
\[ \sum_{i=1}^{m}\langle a_i(T),b_i(s)\rangle = 0\quad\iff\quad s\in T.\]
Then $|S| < R_k(k+1;(k+1)^m)$.
\end{lemma}
We can use essentially the same argument as Averkov to establish the following 
analogue of \cref{lem:averkov-main}, once we replace column spaces by minimal faces and rank with the
largest length of a chain of nonempty faces (minus one). 
We recover Averkov's lemma as a special case by noting that 
the cone $\cS_+^k$ is self-dual and satisfies $\ell(\cS_+^k) = k+1$. 
\begin{lemma}
\label{lem:js-main}
Let $m$ be a positive integer and let $K_1,K_2, \ldots,K_m$ be closed convex cones. 
Let $S$ denote a finite set with cardinality at least $k$. Suppose that, for each $i\in [m]$,
there are maps $a_i:\binom{S}{k}\rightarrow K_i^*$ and $b_i:S\rightarrow K_i$
such that 
\[ \sum_{i=1}^{m}\langle a_i(T),b_i(s)\rangle = 0\quad\iff\quad s\in T.\]
If $\ell(K_i)\leq k+1$ for $i=1,2,\ldots,m$ then $|S| < R_k(k+1;(k+1)^m)$.
\end{lemma}
\begin{proof}
For each $T\subseteq S$ and each $i\in [m]$ define
\[ b_{T,i} := \sum_{t\in T}b_{i}(t)\in K_i\quad\textup{and}\quad d_{T,i} := \ell(F(b_{T,i})).\]
For each $T\in \binom{S}{k}$, we assign `color' $(d_{T,1},\ldots,d_{T,m})$ to the set $T$. Since $\ell(K_i)\leq k+1$, 
then the same is true for any nonempty face of $K_i$, and so  $d_{T,i}\in \{0,1,\ldots,k\}$ 
for each $T$ and $i$. As such, we are coloring with at most $(k+1)^m$ colors.

Seeking a contradiction, let us assume that $|S|\geq R_k(k+1;(k+1)^m)$. 
Then, by the definition of the Ramsey number, there exists 
$W\subseteq S$ with $|W|=k+1$ such that all $k$-element subsets of $W$ have the same color $(d_1,d_2,\ldots,d_m)$.
More explicitly, for all $T\subset W$ with $|T|=k$ we have $(d_{T,1},\dots,d_{T,m}) = (d_1,\ldots,d_m)$. 

\paragraph{Claim} For each $i\in [m]$ and each $T\subset W$ such that $|T|=k$, we have that $F(b_{T,i}) = F(b_{W,i})$. 
If this were not the case, then there exists such an $i\in [m]$ and $T\subset W$ with $|T|=k$ such that 
$F(b_{W,i})$ strictly contains $F(b_{T,i})$.
By \cref{lem:js-subset}, there is a subset $T'\subseteq W$ with $|T'|\leq \ell(F(b_{W,i}))-1 \leq k$ 
such that $F(\sum_{t'\in T'}b_i(t')) = F(b_{W,i})$. 
By adding elements of $W$ (if necessary) we can form a set $T''$ such that $T'\subseteq T''\subset W$ and $|T''|=k$
such that 
\[ F\left(\sum_{t''\in T''}b_i(t'')\right) \supseteq F\left(\sum_{t'\in T'}b_i(t')\right) = F(b_{W,i}).\]
One the one hand, since all $k$-element subsets of $W$ have the same `color', $d_i = d_{T'',i} \geq \ell(F(b_{W,i}))$. 
On the other hand, since $F(b_{W,i})$ strictly contains $F(b_{T,i})$ we have that $\ell(F(b_{W,i})) > \ell(F(b_{T,i})) = d_i$. 
This contradiction establishes the claim.\\[0.2cm]

With the claim established, we write $W = T\cup \{s\}$ where $|T|=k$ and $s\notin T$. 
By assumption on $a_i$ and $b_i$, we have that 
\[ \sum_{i=1}^{m}\langle a_{i}(T),b_{i}(t)\rangle = 0\quad \textup{ for all $t\in T$}.\]
Since $a_{i}(T)\in K_i^*$ and $b_{i}(t)\in K_i$ we can conclude that 
\[ \langle a_{i}(T),b_{i}(t)\rangle = 0\quad\textup{for all $t\in T$ and all $i\in[m]$}.\]
But then 
\[ \left\langle a_{i}(T),\sum_{t\in T}b_{i}(t)\right\rangle = 0.\]
In particular, for each $i\in [m]$ the linear functional defined by $a_{i}(T)$ vanishes on the face
\[ \textstyle{F(\sum_{t\in T}b_{i}(t)) = F(b_{T,i}) = F(b_{W,i})}\]
where the last equality is the claim. 
But then, even though $s\notin T$, we have that
\[ 0 = \left\langle a_i(T),\sum_{w\in W}b_i(w)\right\rangle = \left\langle a_{i}(T), b_i(s) + \sum_{t\in T}b_{i}(t)\right\rangle= \langle a_i(T),b_i(s)\rangle\]
for all $i\in [m]$. This contradicts the fact that $a$ and $b$ satisfy
$0 = \sum_{i=1}^{m}\langle a_i(T),b_i(s)\rangle$
if and only if $s\in T$. It then follows that $|S| < R_k(k+1;(k+1)^m)$. 
\end{proof}

Following Averkov, we combine \cref{lem:averkov-main} with the factorization theorem of Gouveia, Parrilo, and 
Thomas (\cref{thm:gpt}) to turn the 
statement about the structure of maps $a_i$ and $b_i$ into a statement about lifts of convex cones.
We briefly repeat the argument here, in our setting, to make the paper more self-contained.
\begin{proposition}
\label{prop:finitenlift}
Suppose that a proper convex cone $C\subseteq \RR^n$ has a proper 
$K_1\times \cdots\times  K_m$-lift where $\ell(K_i) \leq k+1$ for $i\in [m]$. If $C$ is $k$-neighborly with respect
to some finite set $V\subseteq \textup{ext}(C)$ then $|V|<R_k(k+1;(k+1)^m)$.
\end{proposition}
\begin{proof}
Since $C$ is $k$-neighborly with respect to $V$, for each $k$-element subset $W\in \binom{V}{k}$ there is $f_W\in C^*$
such that $\langle f_W, w\rangle$ vanishes if $w\in W$ and is positive if $w\in V\setminus W$.
Since $C$ has a proper $K_1\times \cdots \times K_m$ lift, by the factorization theorem (\cref{thm:gpt}) we know that
for $i=1,2,\ldots,m$ there exist maps 
$\tilde{a}_i:C^*\rightarrow K_i^*$ and 
$\tilde{b}_i: C \rightarrow K_i$ such that 
$\langle f,x\rangle = \sum_{i=1}^{m}\langle \tilde{a}_i(f),\tilde{b}_i(x)\rangle$ for all 
$f\in C^*$ and all $x\in C$. 
Define $a_i:\binom{V}{k}\rightarrow K_i^*$ by 
$a_i(W) = \tilde{a}_i(f_W)$
and 
$b_i:V \rightarrow K_i$ by $b_i(w) = \tilde{b}_i(w)$.
Then $\langle f_W,w\rangle = \sum_{i=1}^{m}\langle a_i(W),b_i(w)\rangle$
which vanishes when $w\in W$ and is positive when $w\in V\setminus W$. 
It follows from \cref{lem:js-main} that $|V|<R_k(k+1;(k+1)^m)$.
\end{proof}
\Cref{thm:main} is a straightforward consequence of \cref{prop:finitenlift}.
\begin{proof}[{Proof of \cref{thm:main}}]
By assumption, there exists some finite $V\subset \textup{ext}(C)$ such that
$|V|\geq R_k(k+1;(k+1)^m)$ and $C$ is $k$-neighborly with respect to $V$.
Then by the contrapositive of \cref{prop:finitenlift} it follows that $C$ does
not have a proper $K_1\times\cdots\times K_m$ lift with $\ell(K_i) \leq k+1$ for all $i\in [m]$.
It follows from \cref{lem:proper-ced} that $C$ does not have a (possibly
improper) $\tilde{K}_1\times \cdots \times \tilde{K}_m$-lift with
$\ell(\tilde{K}_i)\leq k+1$ for all $i\in [m]$, completing the proof.
\end{proof}
\section{Discussion}
\label{sec:disc}
In this paper we have shown that for a convex cone $C$, having a certain
neighborliness property is an obstruction to having a $K$-lift where $K$ is a
Cartesian product of cones all of which only have short chains of faces.  Our
argument is a fairly direct generalization of an argument of Averkov showing
that the same neighborliness property is an obstruction to having an
$(\cS_+^k)^m$-lift.  

Although we only stated qualitative results about the non-existence of lifts,
the approach taken in this paper could be made quantitative in the following
sense.  If $C$ is $k$-neighborly with respect to a finite set $V$ and
$\ell(K_i)\leq k+1$ for all $i$, then \cref{thm:main} does not rule out the
possibility that $C$ has a $K_1\times \cdots \times K_m$-lift.  However one
could, in principle, extract a lower bound on the number of factors $m$
required in any such lift in terms of $k$ and $|V|$.  Unfortunately, this gives
very weak lower bounds, since the Ramsey numbers $R_{k}(k+1;(k+1)^m)$ grow
extremely fast (and so the implied lower bounds on $m$ grows extremely slowly).
For convex cones that do have $K_1\times \cdots \times K_m$-lifts where
$\ell(K_i)\leq k+1$ for all $i$, it would be very interesting to develop
general techniques to establish stronger lower bounds on $m$, the number of
factors.  In the case $k=1$ this reduces to the study of lower bounds on the
linear programming extension complexity of a polyhedral cone.  

This paper gives some of the first results about non-existence of $K$-lifts
where $K$ is a hyperbolicity cone.  Currently the only other such result is a lower bound on
the size of PSD lifts of convex semialgebraic sets, based on quantifier
elimination~\cite{gouveia2013lifts}, that can be generalized directly to the
setting of hyperbolicity cones.  Fawzi and Safey El Din~\cite{fawzi2018lower}
strengthen the quantifier elimination-based bounds in the case of PSD lifts by
exploiting connections with the algebraic degree of semidefinite
programming~\cite{nie2010algebraic}.  It would be interesting to develop a
similar approach to devise lower bounds on natural notions of complexity for
lifts using hyperbolicity cones.

\paragraph{Acknowledgments} I would like to thank Hamza Fawzi for encouraging and insightful comments related to this work,
Venkat Chandrasekaran for introducing me to the notion of the length of the longest chain of nonempty faces of a convex cone, 
and the anonymous referees for their careful reading and very helpful comments.

\bibliographystyle{alpha}
\bibliography{LB2019-bib}

\begin{thebibliography}{AHMR88}

\bibitem[AB18]{amini2018non}
N.~Amini and P.~Br{\"a}nd{\'e}n.
\newblock Non-representable hyperbolic matroids.
\newblock {\em Advances in Mathematics}, 334:417--449, 2018.

\bibitem[AHMR88]{agler1988positive}
J.~Agler, W.~Helton, S.~McCullough, and L.~Rodman.
\newblock Positive semidefinite matrices with a given sparsity pattern.
\newblock {\em Linear algebra and its applications}, 107:101--149, 1988.

\bibitem[AM19]{ahmadi2019dsos}
A.~A. Ahmadi and A.~Majumdar.
\newblock {DSOS} and {SDSOS} optimization: more tractable alternatives to sum
  of squares and semidefinite optimization.
\newblock {\em SIAM Journal on Applied Algebra and Geometry}, 3(2):193--230,
  2019.

\bibitem[Ave19]{averkov2019optimal}
G.~Averkov.
\newblock Optimal size of linear matrix inequalities in semidefinite approaches
  to polynomial optimization.
\newblock {\em SIAM Journal on Applied Algebra and Geometry}, 3(1):128--151,
  2019.

\bibitem[Bar73]{barker1973lattice}
G.~P. Barker.
\newblock The lattice of faces of a finite dimensional cone.
\newblock {\em Linear Algebra and its Applications}, 7(1):71--82, 1973.

\bibitem[Bar78]{barker1978perfect}
G.~P. Barker.
\newblock Perfect cones.
\newblock {\em Linear Algebra and its Applications}, 22:211--221, 1978.

\bibitem[BC75]{barker1975cones}
G.~P. Barker and D.~Carlson.
\newblock Cones of diagonally dominant matrices.
\newblock {\em Pacific Journal of Mathematics}, 57(1):15--32, 1975.

\bibitem[BSV17]{blekherman2017sums}
G.~Blekherman, R.~Sinn, and M.~Velasco.
\newblock Do sums of squares dream of free resolutions?
\newblock {\em SIAM Journal on Applied Algebra and Geometry}, 1(1):175--199,
  2017.

\bibitem[BTN01]{ben2001lectures}
A.~Ben-Tal and A.~Nemirovski.
\newblock {\em Lectures on modern convex optimization: analysis, algorithms,
  and engineering applications}, volume~2.
\newblock {SIAM}, 2001.

\bibitem[CS16]{chandrasekaran2016relative}
V.~Chandrasekaran and P.~Shah.
\newblock Relative entropy relaxations for signomial optimization.
\newblock {\em SIAM Journal on Optimization}, 26(2):1147--1173, 2016.

\bibitem[Faw18]{fawzi2018representing}
H.~Fawzi.
\newblock On representing the positive semidefinite cone using the second-order
  cone.
\newblock {\em Mathematical Programming}, pages 1--10, 2018.

\bibitem[FSED18]{fawzi2018lower}
H.~Fawzi and M.~Safey El~Din.
\newblock A lower bound on the positive semidefinite rank of convex bodies.
\newblock {\em SIAM Journal on Applied Algebra and Geometry}, 2(1):126--139,
  2018.

\bibitem[G{\aa}r59]{gaarding1959inequality}
L.~G{\aa}rding.
\newblock An inequality for hyperbolic polynomials.
\newblock {\em Journal of Mathematics and Mechanics}, pages 957--965, 1959.

\bibitem[GJSW84]{grone1984positive}
R.~Grone, C.~R. Johnson, E.~M. S{\'a}, and H.~Wolkowicz.
\newblock Positive definite completions of partial {H}ermitian matrices.
\newblock {\em Linear algebra and its applications}, 58:109--124, 1984.

\bibitem[GPT13]{gouveia2013lifts}
J.~Gouveia, P.~A. Parrilo, and R.~R. Thomas.
\newblock Lifts of convex sets and cone factorizations.
\newblock {\em Mathematics of Operations Research}, 38(2):248--264, 2013.

\bibitem[IDW16]{iliman2016amoebas}
S.~Iliman and T.~De~Wolff.
\newblock Amoebas, nonnegative polynomials and sums of squares supported on
  circuits.
\newblock {\em Research in the Mathematical Sciences}, 3(1):9, 2016.

\bibitem[IL17]{ito2017bound}
M.~Ito and B.~F. Louren{\c{c}}o.
\newblock A bound on the {C}arath{\'e}odory number.
\newblock {\em Linear Algebra and its Applications}, 532:347--363, 2017.

\bibitem[KW08]{kalai2008neighborly}
G.~Kalai and A.~Wigderson.
\newblock Neighborly embedded manifolds.
\newblock {\em Discrete \& Computational Geometry}, 40(3):319--324, 2008.

\bibitem[LP18]{liu2018exact}
M.~Liu and G.~Pataki.
\newblock Exact duals and short certificates of infeasibility and weak
  infeasibility in conic linear programming.
\newblock {\em Mathematical Programming}, 167(2):435--480, 2018.

\bibitem[NRS10]{nie2010algebraic}
J.~Nie, K.~Ranestad, and B.~Sturmfels.
\newblock The algebraic degree of semidefinite programming.
\newblock {\em Mathematical Programming}, 122(2):379--405, 2010.

\bibitem[Nui69]{nuij1969note}
W.~Nuij.
\newblock A note on hyperbolic polynomials.
\newblock {\em Mathematica Scandinavica}, 23(1):69--72, 1969.

\bibitem[Pat13]{pataki2013strong}
G.~Pataki.
\newblock Strong duality in conic linear programming: facial reduction and
  extended duals.
\newblock In {\em Computational and analytical mathematics}, pages 613--634.
  Springer, 2013.

\bibitem[Ram30]{ramsey1930problem}
F.~P. Ramsey.
\newblock On a problem of formal logic.
\newblock {\em Proceedings of the London Mathematical Society}, 2(1):264--286,
  1930.

\bibitem[Ren06]{renegar2006hyperbolic}
J.~Renegar.
\newblock Hyperbolic programs, and their derivative relaxations.
\newblock {\em Foundations of Computational Mathematics}, 6(1):59--79, 2006.

\bibitem[Roc15]{rockafellar2015convex}
R.~T. Rockafellar.
\newblock {\em Convex analysis}.
\newblock Princeton University Press, 2015.

\bibitem[VA15]{vandenberghe2015chordal}
L.~Vandenberghe and M.~S. Andersen.
\newblock Chordal graphs and semidefinite optimization.
\newblock {\em Foundations and Trends{\textregistered} in Optimization},
  1(4):241--433, 2015.

\bibitem[WM13]{waki2013facial}
H.~Waki and M.~Muramatsu.
\newblock Facial reduction algorithms for conic optimization problems.
\newblock {\em Journal of Optimization Theory and Applications},
  158(1):188--215, 2013.

\end{thebibliography}

\end{document}